\documentclass[11pt, a4paper]{article}
\usepackage{amsmath, amsthm, amssymb, url, mathtools, enumitem}
\usepackage{multirow, pdflscape, tikz}
\usetikzlibrary{arrows}
\usepackage[utf8]{inputenc}
\renewcommand{\arraystretch}{1.5}

\newcommand{\supp}{\mathrm{supp}}
\newcommand{\Aut}{\mathrm{Aut}}

\newcommand{\proj}{\mathrm{proj}}

\newcommand{\0}{{\bf 0}}
\newcommand{\1}{{\bf 1}}

\theoremstyle{plain}
\newtheorem{theorem}{Theorem}[section]
\newtheorem{assumption*}{Hypothesis}
\newtheorem{theorem*}{Theorem}
\newtheorem{corollary*}{Corollary}

\newtheorem{corollary}[theorem]{Corollary}
\newtheorem{lemma}[theorem]{Lemma}
\newtheorem{proposition}[theorem]{Proposition}
\newtheorem*{claim*}{Claim}

\theoremstyle{definition}
\newtheorem{definition}[theorem]{Definition}
\newtheorem*{notation}{Notation}

\newtheorem{remark}[theorem]{Remark}

\renewcommand{\phi}{\varphi}
\renewcommand{\epsilon}{\varepsilon}

\newcommand\la{\langle}
\newcommand\ra{\rangle}
\newcommand\lla{\la\!\la}
\newcommand\rra{\ra\!\ra}

\setlist[enumerate,1]{label={\upshape\arabic*.}}
\setlist[enumerate,2]{label={\textup{(}\alph*\textup{)}}}

\title{Miyamoto groups of code algebras}

\author{Alonso Castillo-Ramirez\footnote{Departamento de Matematicas, Centro Universitario de Ciencias Exactas e Ingenierias, Universidad de Guadalajara, Mexico, email: alonso.castillor@academicos.udg.mx} \and Justin M\textsuperscript{c}Inroy\footnote{School of Mathematics, University of Bristol, Fry Building, Woodland Road, Bristol, BS8 1UG, UK, and the Heilbronn Institute for Mathematical Research, Bristol, UK, email: justin.mcinroy@bristol.ac.uk}}

\begin{document}
\maketitle

\begin{abstract}
A code algebra $A_C$ is a nonassociative commutative algebra defined via a binary linear code $C$.  In a previous paper, we classified when code algebras are $\mathbb{Z}_2$-graded axial (decomposition) algebras generated by small idempotents.  In this paper, for each algebra in our classification, we obtain the Miyamoto group associated to the grading.  We also show that the code algebra structure can be recovered from the axial decomposition algebra structure. \\ 

\textbf{MSC classes:} 20B25, 17A99, 17D99, 94B05, 17B69.
\end{abstract}

\section{Introduction}

Code algebras were introduced in \cite{codealgebras} as a new class of commutative nonassociative algebras defined from binary linear codes. They contain a family of pairwise orthogonal idempotents, called \emph{toral idempotents}, indexed by the length of the code, and have a nice Peirce decomposition relative to this family. Their definition was inspired by an axiomatic approach to code Vertex Operator Algebras \cite{DGH, M2, codesubVOA}, and, in particular, their connections with axial (decomposition) algebras were explored.

Strengthening the relevance of this axiomatisation, code algebras were shown in \cite{CE19} to have a striking resemblance to the coordinate algebras associated with the optimal short $\text{SL}_2^n$-structure of the simple Lie algebras of types $E_7$, $E_8$, and $F_4$.

In this paper, we continue to explore the links of code algebras with axial decomposition algebras \cite{Axial1}.  These are a relatively new class of commutative nonassociative algebras generated by semisimple idempotents called \emph{axes}. We have partial control over the multiplication by requiring that the eigenvectors for the adjoint action of an axis multiply according to a so-called \emph{fusion law}.  This gives us the key property that, when the fusion law is graded, we may naturally associate automorphisms, called \emph{Miyamoto automorphisms}, to each axis.  These in turn generate the \emph{Miyamoto group} which is a subgroup of the automorphism group.

Let $C \subseteq \mathbb{F}_2^n$ be a binary code of length $n$. A code algebra $A_C$ has a basis $\{t_i : i = 1, \dots, n \} \cup \{e^\alpha : \alpha \in C^* \}$, where $C^* := C  \setminus \{ \0, \1 \}$, and multiplication which mimics the code structure (see Definition \ref{CodeAlgebra} for details).  We call the $t_i$ \emph{toral idempotents} and the $e^\alpha$ \emph{codeword elements}.

In a code algebra $A_C$, the toral idempotents are not enough to generate the algebra. However, in \cite{codealgebras}, we give a construction, called the \emph{$s$-map}, to obtain idempotents of  $A_C$ from a subcode $D$ of $C$.  For the smallest possible subcode, $D = \langle \alpha \rangle$ where $\alpha \in C$, the $s$-map idempotents are called \emph{small idempotents}.  They are of the form
\[
e_{\alpha, \pm} := \lambda t_\alpha \pm \mu e^\alpha
\]
where $t_\alpha = \sum_{i \in \supp(\alpha)} t_i$ and $\lambda, \mu \in \mathbb{F}$.

In general, it is difficult to analyse the eigenvalues and eigenvectors, let alone give the fusion law for $s$-map idempotents. However, in \cite{CM19}, we do this for small idempotents. We showed that if $C$ is a projective code and $S$ is a set of codewords which generate $C$, then the small idempotents with respect to $S$ generate the algebra $A_C$.  We give the parts and evaluation map for the small idempotents explicitly and describe the fusion law in general.  Hence we show when $A_C$ is an axial decomposition algebra with respect to the small idempotents.  The main result in \cite{CM19} classifies when the fusion law for the small idempotents is $\mathbb{Z}_2$-graded under the assumption of the Axis Hypothesis (see page \pageref{axishyp}).  In particular, all $\alpha \in S$ have the same weight $|\alpha|$.

In contrast with \cite{CM19}, where we used the language of axial algebras, in this paper we use the slightly more general language of axial decomposition algebras: this allows us to treat more clearly the cases where different parts have the same eigenvalue. In the first part of this paper, we show that you can recover the code algebra structure from the axial decomposition algebra structure.  We say that a set $X$ of small idempotents is \emph{pair-closed} if whenever $e_{\alpha, +} \in X$, then $e_{\alpha, -} \in X$ and vice versa, for $\alpha \in S$.

\begin{theorem*}
Suppose that $A$ is a code algebra satisfying the Axis Hypothesis for some projective code $C = \la S \ra$ and $X$ be the pair-closed set $X$ of small idempotents.  Then, we can recover the code $C$, up to permutation equivalence, and also the special basis $\{t_i : i = 1, \dots, n \} \cup \{e^\beta : \beta \in  C^*\}$.
\end{theorem*}

In particular, since we know the special basis, we can also recover the parts of the fusion law, even if two different parts happen to have the same eigenvalue.

From the main classification result \cite[Theorem 6.1]{CM19}, we may split the code algebras with a $\mathbb{Z}_2$-graded fusion law up into three cases depending on the weight of $\alpha \in S$ (for a precise statement see Theorem \ref{Z2grading} in Section \ref{sec:codealgebras}):
\begin{enumerate}
\item $|\alpha| = 1$, $C = \mathbb{F}_2^n$, and
\begin{enumerate}
\item $n = 2$, $a=-1$.
\item $n=3$.
\end{enumerate}
\item  $|\alpha| = 2$ and $C = \bigoplus_{i=1}^r C_i$ is the direct sum of even weight codes all of the same length $m \geq 3$.

\item $|\alpha| > 2$
\end{enumerate}

In the second part of the paper, we give an explicit description of the Miyamoto automorphism associated to a small idempotent $e_{\alpha, \pm}$ and calculate the Miyamoto groups.

\begin{theorem*}
Suppose $C=\la S \ra$ is a projective code, $A_C$ is a code algebra satisfying the Axis Hypothesis, and $X$ is the set of pair-closed small idempotents so that $A$ is an axial decomposition algebra with a $\mathbb{Z}_2$-graded fusion law.  Then, the Miyamoto group $G$ of $A$ is
\begin{enumerate}
\item \begin{enumerate}
\item $G = S_3 \times S_3$.
\item $G = 2^2$.
\end{enumerate}
 \item $G = \begin{cases}
2^{r(m-1)} & \mbox{if $m$ is odd,}\\
2^{r(m-2)} & \mbox{if $m$ is even.}
\end{cases}$

\item $G$ is an elementary abelian $2$-group of order at most $2^{|S|}$.
\end{enumerate}
\end{theorem*}

From \cite{CM19}, in case $(2)$, where $|\alpha| = 2$, the algebra is in fact $\mathbb{Z}_2 \times \mathbb{Z}_2$-graded when the structure parameters are chosen in a sufficiently `nice' way (see Section \ref{sec:Z2} for details).

\begin{theorem*}
Suppose $|\alpha|=2$ and $C = \bigoplus_{i=1}^r C_i$ is the direct sum of even weight codes all of the same length $m \geq 3$.  When $A_C$ is an axial decomposition algebra with a $\mathbb{Z}_2 \times \mathbb{Z}_2$-graded fusion law, then $A_C$ has Miyamoto group
\[
G = \begin{cases}
\prod_{i = 1}^r 2^{m-1}:S_m & \mbox{if $m$ is odd}\\
\prod_{i = 1}^r 2^{m-2}:S_m & \mbox{if $m$ is even}
\end{cases}
\]
\end{theorem*}

The structure of the paper is as follows.  We define axial decomposition algebras and describe some of their properties in Section \ref{sec:background}.  We also recall some properties of binary linear codes and recap the definition and main results on code algebras. In Section \ref{sec:params}, the recovery of the code algebra structure from the axial decomposition algebra structure is proved.  Finally, in Section \ref{sec:auts}, we explicitly describe the Miyamoto automorphisms and identify the Miyamoto groups in all cases.

\medskip
We thank the referee for several insightful comments.

\section{Background}\label{sec:background}

We begin by briefly describing axial decomposition algebras and how to obtain automorphisms from a grading.  We then give the definition of a code algebra and recall some results.

\subsection{Axial algebras}\label{sec:axial}
We review the basic definitions related to axial decomposition algebras \cite{decomp}, which are a more general version of axial algebras, see \cite{Axial1, axialstructure}.  Our definition will be a hybrid one.

\begin{definition}
A pair $\mathcal{F} := (\mathcal{F}, \star)$ is a \emph{fusion law} if $\mathcal{F}$ is a non-empty set and $\star \colon \mathcal{F} \times \mathcal{F} \to 2^{\mathcal{F}}$ is a symmetric map, where $2^{\mathcal{F}}$ denotes the power set of $\mathcal{F}$. 
\end{definition}

We assume that our fusion laws have a \emph{unit}, that is, an element $1$ such that $1\star x \subseteq \{x\}$ for all $x \in \mathcal{F}$.  We extend the operation $\star$ to arbitrary subsets $U, V \subseteq \mathcal{F}$ by $U \star V = \bigcup_{\lambda \in U, \mu \in V} \lambda \star \mu$.

Let $A$ be a commutative non-associative (i.e.\ not-necessarily-associative) algebra over $\mathbb{F}$. Let $X$ be a set of elements of $A$.  We write $\la X \ra$ for the vector space spanned by the elements in $X$ and $\lla X \rra$ for the subalgebra generated by $X$.

\begin{definition}\label{axialalgebra}
Let $(\mathcal{F}, \star)$ be a fusion law and $A$ a commutative non-associative algebra over $\mathbb{F}$. An element $a \in A$ is an \emph{$\mathcal{F}$-axis} if the following hold:
\begin{enumerate}
\item $a$ is an \emph{idempotent} (i.e.\ $a^2 = a$),
\item There exists a decomposition of the algebra
\[
A = \bigoplus_{{x} \in \mathcal{F}} A_{x}
\]
such that the algebra multiplication satisfies the fusion law.  That is,
\[
A_x A_y \subseteq A_{x \star y} := \bigoplus_{z \in x \star y} A_z
\]
for all $x,y \in \mathcal{F}$.
\end{enumerate}
If additionally there exists a map $\phi \colon \mathcal{F} \to \mathbb{F}$, which we call an \emph{evaluation map}, such that
\[
a v = \phi(x) v
\]
for all $v \in A_x$, then $a$ is an \emph{${(\mathcal{F}, \phi)}$-axis}.  Furthermore, if $A_{\phi^{-1}(1)} = \langle a \rangle$, then we say $a$ is \emph{primitive}.
\end{definition}

In this paper, we only ever consider axes with evaluation maps, that is $(\mathcal{F}, \phi)$-axes. The above definition is slightly more general than the definition from axial algebras. In the latter, the evaluation map $\phi$ is injective and so we just label the elements of the fusion law by eigenvalues. Our definition allows us to split eigenspaces into different parts and treat them separately.

For $x \in \mathcal{F}$, we call $A_x$ the \emph{$x$-part} of $A$ with respect to $a$. We allow a part $A_x$ to be empty.  Note that we use similar notation $A_\lambda$ for the $\lambda$-eigenspace of $a$. If we want to stress that the part, or eigenspace, is with respect to an axis $a$, we write $A_x(a)$, or $A_\lambda(a)$, but normally we will just write $A_x$.

\begin{definition}
A pair $A = (A, X)$ is an \emph{$\mathcal{F}$-decomposition algebra} if $A$ is a commutative non-associative algebra and $X$ is a set of $\mathcal{F}$-axes which generate $A$.  If $X$ is a set of $(\mathcal{F}, \phi)$-axes, then $A$ is an \emph{$(\mathcal{F}, \phi)$-axial decomposition algebra}.
\end{definition}

When the fusion law is clear from context we drop the $\mathcal{F}$ and $\phi$, and simply use the term \emph{axial decomposition algebra}.

The definition given above is more general than that of an axial algebra, since we allow splitting of eigenspaces via the evaluation map. However, it is not as general as the definition of an axial decomposition algebra introduced in \cite{decomp}. We require our algebra to be commutative, our axes to be idempotents and, most importantly, our axes to generate the algebra. This ensures that, for a given algebra, we cannot arbitrarily add or delete axes.  It also ensures that the Miyamoto group, which we are about to define, is closely associated with the algebra.

\begin{definition}
The fusion law $\mathcal{F}$ is \emph{$T$-graded}, where $T$ is a finite abelian group, if there exist a partition $\{ \mathcal{F}_t : t \in T \}$ of $\mathcal{F}$ such that for all $s,t \in T$,
\[
\mathcal{F}_s \star \mathcal{F}_t \subseteq \mathcal{F}_{st}
\]
\end{definition}
We allow the possibility that some part $\mathcal{F}_t$ is the empty set. Let $A$ be an algebra and $a \in A$ an $(\mathcal{F}, \phi)$-axis (we do not require that $A$ be an axial decomposition algebra).  If $\mathcal{F}$ is \emph{$T$-graded}, then the axis $a$ defines a \emph{$T$-grading} on $A$ where the $t$-graded subspace $A_t$ of $A$ is
\[
A_t = \bigoplus_{x \in \mathcal{F}_t } A_{x}(a)
\]
When $\mathcal{F}$ is $T$-graded we may define some automorphisms of the algebra in the following way.  Let $T^*$ denote the set of linear characters of $T$ (i.e., the set of all homomorphisms from $T$ to $\mathbb{F}^\times$). For an axis $a$ and $\chi\in T^*$, consider the map $\tau_a(\chi)\colon A\to A$ defined by the linear extension of
\[
u \mapsto \chi(t) u \qquad \mbox{for } u \in A_t(a).
\]
Since $A$ is $T$-graded, this map $\tau_a(\chi)$ is an automorphism of $A$, which we call a \emph{Miyamoto automorphism}. Furthermore, the map sending $\chi$ to $\tau_a(\chi)$ is a homomorphism from $T^*$ to $\Aut(A)$.

The subgroup $T_a := \langle \tau_a(\chi) : \chi \in T^* \rangle$ is called the \emph{axial subgroup} corresponding to $a$.
\begin{definition}
Let $X$ be a set of $(\mathcal{F}, \phi)$-axes.  The \emph{Miyamoto group with respect to $X$} is
\[
G(X) := \langle T_a : a \in X \rangle  \leq \Aut(A)
\]
\end{definition}

We are particularly interested in $\mathbb{Z}_2$-graded fusion laws. In this case, we identify $\mathbb{Z}_2$ with the group $\{ +, - \}$ equipped with the usual multiplication of signs.  When $\mathrm{char}(\mathbb{F}) \neq 2$, the sign character $\chi_{-1}$, where $\pm \mapsto \pm 1$, is the only non-trivial character.  We write $\tau_a := \tau_a(\chi_{-1})$ and call it the \emph{Miyamoto involution associated to $a$}.

A set of axes $X$ is called \emph{closed} if it is closed with respect to its Miyamoto group.  That is, $X = X^{G(X)}$. One can show that given any set of axes $X$, there is a unique smallest closed set $\bar{X} \supseteq X$, called the \emph{closure} of $X$, such that $\bar{X}$ is closed.  By \cite[Lemma 3.5]{axialstructure}, $\bar{X} = X^{G(X)}$ and $G(\bar{X}) = G(X)$.

\subsection{Codes}

A \emph{binary linear code} $C$ of \emph{length} $n$ and \emph{rank} $k$ is a $k$-dimensional subspace of $\mathbb{F}_2^n$ and we call elements of $C$ \emph{codewords}.  We will write $\0$ for the all zeroes codeword $(0, \dots, 0)$ and $\1$ for the all ones codeword $(1, \dots, 1)$.  Note that a code $C$ always contains $\0$, but does not necessarily contain $\1$.  If $\1 \in C$, then each codeword $\alpha \in C$ has a \emph{complement} $\alpha^c := \alpha +\1$.  Conversely, if $\alpha \in C$ has a complement, that is there exists a vector $\alpha^c \in C$ such that $\alpha+\alpha^c = \1$, then $\1 \in C$ and all codewords have complements.

We denote the \emph{support} of $\alpha \in \mathbb{F}_2^n$ by $\supp(\alpha) := \{ i = 1, \dots, n : \alpha_i = 1\}$ and say it has \emph{Hamming weight} $|\alpha| := |\supp(\alpha)|$.  We denote the set of weights in $C$ by $wt(C) := \{ |\alpha| : \alpha \in C \}$.  For two codewords $\alpha, \beta \in C$, we use the notation $\alpha \cap \beta := \supp(\alpha) \cap \supp(\beta)$.

Two codes $C$ and $D$ of length $n$ are \emph{permutation equivalent} if there exists a permutation $\sigma \in S_n$ such that $D = C^\sigma$, where $\sigma$ acts naturally on codewords by permuting their support.

The \emph{dual code} $C^\perp$ of a code $C$ is the set of $v \in \mathbb{F}_2^n$ such that $(v,C) = 0$ where $(\cdot, \cdot)$ is the usual dot product.  A binary linear code $C$ is \emph{projective} if the minimum weight of a codeword in its dual $C^\perp$ is at least three.

For a subset $S$ of $\{1, \dots, n\}$, we write $\proj_S\colon C \to \mathbb{F}_2^n$ for the usual projection map onto the subspace of $\mathbb{F}_2^n$ spanned by the standard basis vectors $e_i = (0, \dots, 0,1,0,\dots, 0)$ with $i \in S$.  Note that $\proj_S(C)$ is also a binary linear code.  In more code theoretic language, it is the code \emph{punctured at $[n] - S$}, where $[n]:=\{ 1,2, \dots, n \}$.  For $\alpha \in C$, we will abuse our notation and write $\proj_\alpha(C)$ for $\proj_{\supp(\alpha)}(C)$.

\subsection{Code algebras}\label{sec:codealgebras}

We write $C^* := C \setminus \{ \0, \1 \}$.  Throughout this section and for the rest of the paper, we will always assume that the field $\mathbb{F}$ is large enough to contain any square roots we need (we will indicate in the text which quadratic equations these come from).  In particular, this will only ever be a finite number and hence taking a finite field extension will suffice.

\begin{notation}
Note that throughout this paper, we will often write conditions involving $\1$, or $\alpha^c$.  We do not assume that $\1 \in C$, or complements exist, just that if they do, then the conditions must hold.
\end{notation}

\begin{definition}\label{CodeAlgebra}
Let $C \subseteq \mathbb{F}_2^n$ be a binary linear code of length $n$, $\mathbb{F}$ a field and $\Lambda \subseteq \mathbb{F}$ be a collection of structure parameters
\[
\Lambda := \left\{ a_{i,\alpha}, b_{\alpha,\beta}, c_{i,\alpha} \in \mathbb{F}  : i \in \supp(\alpha), \alpha, \beta \in C^*, \beta \neq \alpha, \alpha^c\right\}.
\]
The \emph{code algebra} $A_C(\Lambda)$ is the commutative algebra over $\mathbb{F}$ with basis
\[
\{ t_i : i = 1, \dots, n \} \cup \{ e^{\alpha} : \alpha \in C^* \},
\]
and multiplication given by
\begin{align*}
t_i \cdot t_j & = \delta_{i,j} t_i \\
t_i \cdot e^\alpha & = \begin{cases} 
a_{i,\alpha} \, e^\alpha & \text{if } \alpha_i = 1 \\
\mathrlap0\phantom{ \sum \limits_{i \in \supp(\alpha) }c_{i,\alpha} t_i} & \text{if } \alpha_i =0
\end{cases} \\
e^\alpha \cdot e^\beta & = \begin{cases}
b_{\alpha, \beta}\, e^{\alpha + \beta} & \text{if } \alpha \neq \beta, \beta^c \\
 \sum \limits_{i \in \supp(\alpha) }c_{i,\alpha} t_i & \text{if } \alpha = \beta  \\
0 & \text{if } \alpha = \beta^c
\end{cases}
\end{align*}
where $\delta_{i,j}$ is the Kronecker delta.
\end{definition}

A code algebra is \emph{non-degenerate} if all of its structure parameters are non-zero. We call the basis elements $t_i$ \emph{toral elements} and the $e^\alpha$ \emph{codeword elements}.  For $\alpha \in C^*$, we use the notation $t_\alpha$ for $\sum_{i \in \supp(\alpha)} t_i$.

Given a subcode $D$ of $C$, the $s$-map construction can be used for defining idempotents.  This was first given in \cite{codealgebras} and subsequently revised in \cite[Proposition 2.2]{CM19}.  In this paper, we will be interested in the so-called \emph{small idempotents} coming from this $s$-map construction. Fix $\alpha \in C^*$ and assume that $a_{\alpha} := a_{i,\alpha} = a_{j,\alpha}$ and $c_\alpha := c_{i,\alpha} = c_{j,\alpha}$, for all $ i,j \in \supp(\alpha)$.  Then for the subcode $D := \langle \alpha \rangle$, the $s$-map construction gives two idempotents
\[
e_{\alpha, \pm} := \lambda_{\alpha} t_\alpha \pm \mu_\alpha e^\alpha
\]
where $\lambda_{\alpha} := \frac{1}{2 a_{\alpha} |\alpha|}$ and $\mu_\alpha^2 = \frac{\lambda_{\alpha} - \lambda_{\alpha}^2}{c_\alpha}$.  In particular, we assume that the roots $\pm \mu_\alpha$ exist in $\mathbb{F}$.

The idempotents are particularly interesting when they lead to automorphisms of the algebra.  As described in Section \ref{sec:axial}, if the idempotents are in fact axes for some graded fusion law then they have naturally associated Miyamoto automorphisms.  The situation where the code algebra is an axial decomposition algebra and the fusion law is $\mathbb{Z}_2$-graded was characterised in \cite{CM19}. We shall briefly recap the assumptions we need for this below, referring the reader to \cite{CM19} for the full details.  Note that Definition \ref{intregular} and the Field Hypothesis below are new, but they just better organise the assumptions.

\begin{definition}
Given $\beta \in C^*$, we define the \emph{weight partition of $\beta$ with respect to $\alpha$} to be the unordered pair of integers
\[
p(\beta) :=\left( | \alpha \cap \beta|, | \alpha \cap (\alpha + \beta)| \right) = (\vert \alpha \cap \beta \vert, \vert \alpha \cap \beta^c \vert).\]
Note that $p(\beta) = p( \alpha + \beta) = p( \beta^c)$. For any pair of integers $p:= (a,b)$ such that $a+b = \vert \alpha \vert$, define
\[ C_\alpha(p) := \{ \beta \in C^*\setminus \{\alpha, \alpha^c\} : p(\beta) = p \}; \]
in other words, this is the set of all $\beta$ which give the weight partition $p$. We define
\[
P_\alpha := \{ p(\beta) : \beta \in C^* \setminus \{ \alpha, \alpha^c\} \}
\]
to be the set of all weight partitions with respect to $\alpha$.
\end{definition}

Since the different parts for an axis $e_{\alpha, \pm}$ will correspond to weight partitions $p \in P_\alpha$, we must restrict the structure constants on the weight partitions.

\begin{definition}\label{intregular}
Let $S$ be a set of codewords in $C$. A non-degenerate code algebra $A_C$ is \emph{$S$-intersection regular} if for all $\alpha, \alpha' \in S$, we have $P_\alpha = P_{\alpha'}$ and 
\begin{align*}
a  &:= a_{i, \beta} &&\mbox{for all } i \in \supp(\beta), \beta \in C^*\\
b_{\alpha, \beta} &= b_{\alpha, \gamma} &&\mbox{for all } \beta \in C_{\alpha}(p), \gamma \in C_{\alpha'}(p), p \in P_\alpha \\
b_{\alpha^c, \beta} &= b_{\alpha'^c, \gamma} &&\mbox{for all } \beta \in C_{\alpha}(p), \gamma \in C_{\alpha'}(p), p \in P_\alpha \\
c_\beta &:= c_{i, \beta} &&\mbox{for all } i \in \supp(\beta), \beta \in C^*
\end{align*}
\end{definition}

In particular, this implies that $|\alpha| = |\alpha'|$ for all $\alpha, \alpha' \in S$, since their set of weight partitions are the same.  Moreover, since $\lambda_\alpha = \frac{1}{2a_\alpha|\alpha|}$ does not depend on $\alpha$, for the rest of the paper, we will just write $\lambda$ for $\lambda_\alpha$.

From now on, suppose that $S$ is a set of codewords of $C$, $A_C$ is a non-degenerate $S$-intersection regular code algebra and $\alpha \in S$.

Since we want our idempotent to be an axis with an evaluation map, we need to take the field large enough for the eigenvalues and eigenvectors to exist.  To do this, we need to further define some scalars which will be the coefficients of the eigenvalues and eigenvectors. For $\beta \in C^* \setminus \{ \alpha, \alpha^c \}$, we define
\[
\xi_\beta := \frac{\lambda a}{2\mu_\alpha b_{\alpha, \beta}}(|\alpha| - 2|\alpha \cap \beta|) = \frac{1 }{4 \mu_\alpha b_{\alpha, \beta}} \left(1 - \frac{2 \vert \alpha \cap \beta\vert}{\vert \alpha \vert } \right)
\]
and let $\theta^\pm_\beta$ be the two roots of
\[
x^2 + 2\xi_\beta x -1 = 0.
\]
Note that since $A_C$ is $S$-intersection regular, the above definitions of $\xi_\beta$ and $\theta^\pm_\beta$ are independent of the choice of $\alpha$.  For the small idempotent $e_{\alpha, -}$, we can put $-\mu_\alpha$ instead of $\mu_\alpha$ in the above and see that $\xi_\beta$ and $\theta^\pm_\beta$ are both multiplied by $-1$.  Since this does not alter the Field Hypothesis below and will cancel later, we do not include it in the notation here.

\begin{assumption*}[Field Hypothesis]\label{fieldassumption}
Assume $\mathbb{F}$ is a field of characteristic $p \neq 2$ and not dividing $|\alpha|$, where $\alpha \in S$ \textup{(}allowing $p=0$\textup{)}, and such that $\xi_\beta^2 \neq -1$ for all $\beta \in C^* \setminus \{ \alpha, \alpha^c \}$.
\end{assumption*}

Since $\xi_\beta$ does not depend on $\alpha$, the above hypothesis is well-defined.  Note that the condition $\xi_\beta^2 \neq -1$ is equivalent to the two roots $\theta^\pm_\beta$ of the quadratic above being distinct.

The following will be our running assumption throughout the paper.

\begin{assumption*}[Axis Hypothesis]\label{axishyp}
Assume $A_C$ is an $S$-intersection regular code algebra over a field $\mathbb{F}$ satisfying the Field Hypothesis and such that $a \neq \frac{1}{2|\alpha|}, \frac{1}{3|\alpha|}$.
\end{assumption*}

Note that if $a=\frac{1}{2|\alpha|}$, then $\lambda = 1$ and $\mu = 0$ and so $e_{\alpha, \pm}$ is just a sum of toral idempotents.  If $a=\frac{1}{3|\alpha|}$, then \cite[Lemma 3.7]{CM19} implies that $e_{\alpha, \pm}$ is not an $(\mathcal{F}, \phi)$-axis for any evaluation map $\phi$ as its adjoint action is not semisimple.  Since we are interested in new axes with evaluations, we rule out both these cases in the Axis Hypothesis.

Finally, we define
\[
\nu^\pm_p := \tfrac{1}{4} + \mu_\alpha b_{\alpha, \beta} (\theta^\pm_\beta + \xi_\beta)
\]
for $p \in P_\alpha$, $\beta \in C_\alpha(p)$. Notice that this does not depend on $\alpha \in S$ and is the same for $e_{\alpha,+}$ and $e_{\alpha, -}$. We may now state the following.

\begin{theorem}[{\cite[Proposition 3.5, Theorem 4.1]{CM19}}]
Suppose $A_C$ is a code algebra satisfying the Axis Hypothesis with respect to $S = \{\alpha \}$.  Then, $e = e_{\alpha, +}$ is an axis with fusion law $\mathcal{F}$ given in Table $\ref{tab:small}$, parts given explicitly in Table $\ref{tab:esp}$ and evaluation map $\phi$ given by
\begin{align*}
x &\mapsto x && \mbox{where } x = 1, 0, \lambda, \lambda-\tfrac{1}{2}\\
p^\epsilon &\mapsto \nu^\epsilon_p && \mbox{where } \epsilon = \pm
\end{align*}
Moreover, it is a primitive axis if $\phi(p^\pm) = \nu^\pm_p \neq 1$ for all $p \in P_\alpha$.
\end{theorem}

In the above theorem, the parts in Table $\ref{tab:esp}$ are for $e_{\alpha, +}$; for the analogous result for $e_{\alpha,-}$ simply replace $\mu$ by $-\mu$.

Note that in the fusion law, $p^+$ and $p^-$, for $p \in P_\alpha$, are separate elements, but since they behave in a very similar way, we abuse notation and join their columns in Table \ref{tab:small} in a single colonnade $p^\pm$.

\begin{table}[!htb]
\setlength{\tabcolsep}{4pt}
\renewcommand{\arraystretch}{1.5}
\centering
\begin{tabular}{c|ccccccc}
 & $1$ & $0$ & $\lambda$ & $\lambda-\frac{1}{2}$ & $p_1^\pm$ &  $\dots$ & $p_k^\pm$ \\ \hline
$1$ & $1$ &  & $\lambda$ & $\lambda-\frac{1}{2}$ & $p_1^\pm$ &  $\dots$ & $p_k^\pm$ \\
$0$ &  & $0$ &  &  & $p_1^\pm$ &  $\dots$ & $p_k^\pm$ \\
$\lambda$ & $\lambda$ &  & $1, \lambda, \lambda-\frac{1}{2}$ &  & $p_1^+, p_1^-$  & \dots & $p_k^+, p_k^-$ \\
$\lambda-\frac{1}{2}$ & $\lambda-\frac{1}{2}$&  &  & $1, \lambda-\frac{1}{2}$ & $p_1^+, p_1^-$ & \dots & $p_k^+, p_k^-$ \\
$p_1^\pm$ & $p_1^\pm$ & $p_1^\pm$  & $p_1^+, p_1^-$ & $p_1^+, p_1^-$  & $X_1$ && $N(p_1, p_k)$ \\
$\vdots$ & $\vdots$ &$\vdots$ &$\vdots$ &$\vdots$ & & $\ddots$ &\\
$p_k^\pm$ & $p_k^\pm$ & $p_k^\pm$  & $p_k^+, p_k^-$ & $p_k^+, p_k^-$ & $N(p_k, p_1)$ & & $X_k$
\end{tabular}
\vspace{5pt}

where
\[
N(p,q) := \{ p(\beta + \gamma)^+, p(\beta + \gamma)^- : \beta \in C_\alpha'(p), \gamma \in C_\alpha'(q), \gamma \neq \beta, \alpha+\beta, \beta^c, \alpha+\beta^c \}
\]

\vspace{5pt}
and $X_i$ represents the table

\vspace{5pt}
\begin{tabular}{c|cc}
& $p_i^+$ & $p_i^-$ \\
\hline
$p_i^+$ & $1,0,\lambda, \lambda-\frac{1}{2}, N(p_i, p_i)$ & $1,0,\lambda, \lambda-\frac{1}{2}, N(p_i, p_i)$ \\
$p_i^-$ & $1,0,\lambda, \lambda-\frac{1}{2}, N(p_i, p_i)$ & $1,0,\lambda, \lambda-\frac{1}{2}, N(p_i, p_i)$
\end{tabular}

\caption{Fusion law $\mathcal{F}$ for small idempotents}\label{tab:small}
\end{table}

\begin{table}[!htb]
\setlength{\tabcolsep}{4pt}
\renewcommand{\arraystretch}{1.5}
\centering
\begin{tabular}{l|l@{\hskip 15pt}l}
Part & Basis \\
\hline
$1$ & $e = \lambda t_\alpha + \mu_\alpha e^\alpha$  \\
\hline
\multirow{2}{4em}{$0$} &  $t_i$ & for $i \not \in \supp(\alpha)$ \\
& $e^{\alpha^c}$ \\
\hline
$\lambda$ & $t_j - t_k$ & for $k \in \supp(\alpha)$, $k \neq j := \min(\supp(\alpha))$ \\
\hline
$\lambda - \tfrac{1}{2}$ & $2\mu_\alpha c_\alpha t_\alpha - e^\alpha$ \\
\hline
$p^\pm$ & $w^\pm_\beta = \theta^\pm_\beta e^\beta + e^{\alpha + \beta}$ & for $\beta \in C_\alpha'(p)$, $p \in P_\alpha$
\end{tabular}
\caption{Parts for small idempotents}\label{tab:esp}
\end{table}

Observe that the fusion law depends only on the weight partition $P_\alpha$ of $\alpha$ (even though the parts depend on the value of $\mu_\alpha$ and therefore on $c_\alpha$).  However, the evaluation map and the eigenvalues do depend on $c_\alpha$.

In \cite[Theorem 5.1]{CM19}, we implicitly assumed that $c_{\alpha} = c_{\alpha'}$ for all $\alpha, \alpha' \in S$. However, two evaluation maps can be the same for different values of $c_\alpha$, so in the theorem below we just assume that all the evaluation maps are the same.  In the next section, we show that, in all but one small case where $C \neq \mathbb{F}_2$, we must indeed have $c = c_{\alpha}$, for all $\alpha\in S$, for the fusion law to be $\mathbb{Z}_2$-graded.

\begin{theorem}[{\cite[Theorem 5.1]{CM19}}]\label{codeisaxialalg}
Suppose $C = \la S \ra$ is a projective code and $A_C$ is a code algebra satisfying the Axis Hypothesis for $S$.  Let $X$ be the set of small idempotents with respect to $S$.  Then $A_C(\Lambda)$ is an $\mathcal{F}$-decomposition algebra with fusion law $\mathcal{F}$ given in Table $\ref{tab:small}$.

If in addition all the evaluation maps $\phi_e$, for $e \in X$, are the same, then $A_C(\Lambda)$ is an $(\mathcal{F}, \phi)$-axial decomposition algebra.
\end{theorem}

In this paper we are interested in calculating the associated Miyamoto groups.  These are non-trivial if and only if the fusion law is graded. Recall that the weight set of a code $D$ is the set of weights $wt(D) = \{ |\beta| : \beta \in D \}$.

\begin{theorem}[{\cite[Theorem 6.1]{CM19}}]\label{Z2grading}
Assume the Hypotheses of Theorem $\ref{codeisaxialalg}$ and that the evaluation maps are the same. Let $\alpha \in S$ and define $D = D(\alpha) := \proj_\alpha(C)$. Then the fusion law of $e_{\alpha, \pm} \in X$ is $\mathbb{Z}_2$-graded if and only if
\begin{enumerate}
\item $|\alpha| = 1$, $C = \mathbb{F}_2^n$, and
\begin{enumerate}
\item $n = 2$, $a=-1$.
\item $n=3$.
\end{enumerate}

\item  $|\alpha| = 2$ and $C = \bigoplus C_i$ is the direct sum of even weight codes all of the same length $m \geq 3$.

\item $|\alpha| > 2$ where $D = D(\alpha)$ is a projective code, $\1 \in D$ and there exists a codimension one linear subcode $D_+$ of $D$ such that $ \1 \in D_+$ and $D_+$ is the union of weight sets of $D$.
\end{enumerate}
In cases $(2)$ and $(3)$, we have
\begin{align*}
A_+ &= A_1 \oplus A_0 \oplus A_\lambda \oplus A_{\lambda - \frac{1}{2}} \oplus
\bigoplus_{m \in wt(D_+)} A_{(m, |\alpha|- m)^\pm} \\
A_- &= \bigoplus_{m \in wt(D) - wt(D_+)} A_{(m, |\alpha|- m)^\pm}
\end{align*}
where in case $(2)$, $D = D(\alpha) = \mathbb{F}_2^2$ and $D_+ = \{ \0, \1 \}$, for all $\alpha \in S$.
\end{theorem}

Note that in the second case, provided we make some further assumptions on the structure parameters, the fusion law is in fact $\mathbb{Z}_2 \times \mathbb{Z}_2$-graded, which leads to further Miyamoto automorphisms.  This situation will be described later where we calculate the associated Miyamoto group.


\section{Structure parameters}\label{sec:params}

Suppose that $A$ is an axial decomposition algebra with a $\mathbb{Z}_2$-graded fusion law which we know to also be a code algebra coming from a projective code $C$ and the axes $X$ are a set of small idempotents.  Can we recover the code algebra structure?  That is, we do not assume that the code, or the special basis of the toral and codeword elements are known.

We say that a set $X$ of small idempotents is \emph{pair-closed} if whenever $\lambda t_\alpha + \mu_\alpha e^\alpha \in X$, then $\lambda t_\alpha - \mu_\alpha e^\alpha \in X$ and vice versa.

\begin{theorem}\label{recover}
Suppose that $A$ is a code algebra satisfying the Axis Hypothesis for some projective code $C = \la S \ra$ and $X$ be the pair-closed set of small idempotents with respect to $S$.  Then, we can recover the code $C$, up to permutation equivalence, and also the special basis $\{t_i : i = 1, \dots, n \} \cup \{e^\alpha : \alpha \in C^*\}$.
\end{theorem}

Note that in an axial decomposition algebra we can easily calculate the eigenspaces given any axis.  However, if the evaluation map $\phi$ is not injective, then we may not be able to calculate the parts.  In particular, it may be the case that, for distinct $x, y \in \mathcal{F}$, $\phi(x) = \phi(y)$ and $x \star z = y \star z$ for all $z \in \mathcal{F}$.  Then, the $x$- and $y$-parts would be indistinguishable from the algebra structure.

\begin{corollary}
Suppose that $A$ is a code algebra satisfying the Axis Hypothesis for some projective code $C = \la S \ra$ and $X$ is the pair-closed set of small idempotents with respect to $S$. Then the parts $x \in \mathcal{F}$ are distinguishable.
\end{corollary}
\begin{proof}
By Theorem \ref{recover}, we know the special basis and hence the parts.
\end{proof}

We prove Theorem \ref{recover} via a series of lemmas.
\medskip

First we reduce the problem to partitioning the set $X$ into pairs of idempotents that come from the same codeword.

Suppose that we have such a partition; our pairs are $\{\lambda t_\alpha \pm \mu_\alpha e^\alpha \}$ with $\alpha \in S$ for some set $S$.  We stress that the $\alpha$ in $S$ are just formal labels as the code is not yet known.  We show that we can indeed recover the code $C$ from this.

\begin{proposition}
Suppose that we can partition the set $X$ of axes into pairs $\{\lambda t_\alpha \pm \mu_\alpha e^\alpha \}$ with $\alpha \in S$ for some set $S$.  Then, we can recover the code $C$, up to permutation equivalence, and we know the special basis $\{t_i : i = 1, \dots, n \} \cup \{e^\beta : \beta \in C\}$ \textup{(}up to scaling for the $e^\beta$\textup{)}.
\end{proposition}
\begin{proof}
Given each pair $e_{\alpha, \pm} = \lambda t_\alpha \pm \mu_\alpha e^\alpha$, we sum them to recover $t_\alpha$ and subtract them to recover $e^\alpha$ (up to scalars).  In particular, we know which is the toral and which is the codeword element.  Since $X$ generates the algebra, $Y := \{ t_\alpha, e^\alpha : \alpha \in S \}$ also generates the algebra.

We claim that $S$ generates the code.  Indeed, since $Y$ generates the algebra, we may write $e^\beta$ as a sum of products in $Y$.  However, it is clear from Definition \ref{CodeAlgebra} that we may do this just using codeword elements and we only need one summand.  That is $e^\beta$ is some scalar multiple of $e^{\alpha_1} \dots e^{\alpha_k}$ for some $\alpha_1, \dots \alpha_k \in S$.  Hence $\beta = \alpha_1 + \dots +\alpha_k$ and $S$ generates the code.

Let $T = \{ t_i : i = 1, \dots, n \}$ be the set of toral elements.  Since $S$ generates the code and the code is projective, we may multiply the $t_\alpha$ to get any $t \in T$.  Explicitly, we know that the code we are trying to reconstruct is projective.  Hence, for any $i \in 1, \dots, n$, there exist a subset of codewords $Y \subseteq S$ such that $\{i \} = \bigcap_{\beta \in Y} \supp(\beta)$.  Note that $t_\alpha t_\beta = t_{\alpha \cap \beta}$ and so multiplication of the $t_\alpha$ corresponds to intersection of the corresponding codewords.  Hence $\lla t_\alpha : \alpha \in S \rra = \lla T \rra$.  The unique set of pairwise annihilating idempotents in $\lla T \rra$ is $T$.  Note however, that since $S$ is just a set of formal labels, we can only recover $T$ up to a permutation of $\{1, \dots, n\}$.

By above, we can construct each $e^\beta$ (up to scaling) and so we can recover the special basis.  Therefore we know the code $C$, as its additive structure is the multiplicative structure of the set $\{ e^\beta : \beta \in C \}$.  Note however, as argued above with the $t_i$, we only know $C$ up to a permutation action on the columns.
\end{proof}

We now show that we can indeed find the pairs.

\begin{lemma}\label{1stpair}
Given $\lambda t_\alpha + \mu_\alpha e^\alpha \in X$, we can identify $\lambda t_\alpha - \mu_\alpha e^\alpha \in X$ provided we are not in the case where $\1 \in C$, $|\alpha| = \frac{n}{2}$, $a = \frac{2}{n} = \frac{1}{|\alpha|}$, $\lambda = \frac{1}{2}$ and $\lambda t_{\alpha^c} - \mu_{\alpha^c} e^{\alpha^c} \in X$.
\end{lemma}
\begin{proof}
We consider subalgebras generated by pairs of axes and study properties of such a subalgebra.  First consider the case that our pair is $\{\lambda t_\alpha \pm \mu_\alpha e^\alpha \}$ as wanted.  Note that
\begin{align*}
(\lambda t_\alpha + \mu_\alpha e^\alpha)(\lambda t_\alpha - \mu_\alpha e^\alpha) &= (\lambda^2 - \mu_\alpha^2 c_\alpha)t_\alpha\\
&= \lambda(2\lambda -1) t_\alpha
\end{align*}
which is non-zero when $\lambda \neq \frac{1}{2}$ or equivalently $a \neq \frac{1}{|\alpha|}$.  So in all cases $\lla \lambda t_\alpha \pm \mu_\alpha e^\alpha \rra$ is $2$-dimensional, but in particular the two axes are not mutually orthogonal unless $\lambda = \frac{1}{2}$.

Suppose that our possible pair is $\{ \lambda t_\alpha + \mu_\alpha e^\alpha, \lambda t_\beta + \mu_\beta e^\beta \}$ for some $\beta \in C^* - \{ \alpha, \alpha^c\}$.  Then,
\[
(\lambda t_\alpha + \mu_\alpha e^\alpha)(\lambda t_\beta + \mu_\beta e^\beta) = \lambda^2 t_{\alpha \cap \beta} + \lambda a|\alpha \cap \beta|(\mu_\alpha e^\alpha + \mu_\beta e^\beta) + \mu_\alpha \mu_\beta b_{\alpha, \beta} e^{\alpha + \beta}
\]
In particular, since $\mu_\alpha \mu_\beta b_{\alpha, \beta} \neq 0$ and $\alpha + \beta \neq \alpha, \beta$, the above is not in the span of $\{ \lambda t_\alpha + \mu_\alpha e^\alpha, \lambda t_\beta + \mu_\beta e^\beta \}$.  Hence, the dimension of $\lla \lambda t_\alpha + \mu_\alpha e^\alpha, \lambda t_\beta + \mu_\beta e^\beta \rra$ is strictly bigger than two.

Finally, we are left with the case where the pair is $\{ \lambda t_\alpha + \mu_\alpha e^\alpha, \lambda t_{\alpha^c} + \mu_{\alpha^c} e^{\alpha^c} \}$. (This can only happen $\lambda t_{\alpha^c} + \mu_{\alpha^c} e^{\alpha^c} \in X$ which implies that $|\alpha^c| = |\alpha| = \frac{n}{2}$ and $\1 \in C$.) The pair are mutually orthogonal idempotents and hence generate a subalgebra of dimension two.  Hence, unless $\1 \in C$, $|\alpha| = \frac{n}{2}$ and $\lambda = \frac{1}{2}$, we can correctly pair partition $X$ into pairs $\{\lambda t_\alpha \pm \mu_\alpha e^\alpha \}$.
\end{proof}

So, we can reduce to the case where $\1 \in C$, $|\alpha| = \frac{n}{2}$, $a = \frac{2}{n} = \frac{1}{|\alpha|}$ and $\lambda = \frac{1}{2}$.  Given $e_+ := \lambda t_\alpha + \mu_\alpha e^\alpha$ it remains to distinguish between $e_- := \lambda t_\alpha - \mu_\alpha e^\alpha$ and $e^c_\pm := \lambda t_{\alpha^c} \pm \mu_{\alpha^c} e^{\alpha^c}$ to form the correct pair.

We consider the parts of $e_+$.  First note that the $\lambda$-part is spanned by $t_j-t_k$ for $j,k \in \supp(\alpha)$.  So $e_+$ and $e_-$ share a $\lambda$-part, whilst the $\lambda$-parts of $e_+$ and $e^c_\pm$ are disjoint. Hence if the $\lambda$-part of an axis is distinguished, then we are done.

If the $\frac{1}{2}$-eigenspace just consists of the $\lambda$-part, then it is clearly distinguished.  So, assume that there is some other part which is also in the $\frac{1}{2}$-eigenspace.  Since it cannot be $1$, $0$, or $\lambda-\frac{1}{2}$, it must be $p^\epsilon$ for some $p \in P_\alpha$, $\epsilon = \pm$.

Next we consider the $0$-eigenspace.  Since $\lambda = \frac{1}{2}$, $\phi(\lambda-\frac{1}{2}) = 0$ and so necessarily both the $0$- and $(\lambda-\frac{1}{2})$-parts are in the $0$-eigenspace.

\begin{lemma}
If we can distinguish the sum of the $0$- and $(\lambda-\frac{1}{2})$-parts from the remainder of the the $0$-eigenspace, then the $\lambda$-part is distinguished
\end{lemma}
\begin{proof}
Note that $0 \star \lambda = \emptyset$ and $(\lambda-\frac{1}{2}) \star \lambda = \emptyset$.  However, since $\1 \in C$, an easy calculation (\cite[Lemma 4.6]{CM19}) shows that $\emptyset \neq 0 \star p^\epsilon = p^\epsilon$.  Similarly, since $a = \frac{1}{|\alpha|}$, $\emptyset \neq (\lambda - \frac{1}{2}) \star p^\epsilon = p^\epsilon$ (\cite[Lemma 4.11]{CM19}).  Hence, the $\lambda$-part is distinguished from the $p^\epsilon$-part in the $\frac{1}{2}$-eigenspace.
\end{proof}

It remains to show the following.

\begin{lemma}
We can distinguish the sum of the $0$- and $(\lambda-\frac{1}{2})$-parts in the $0$-eigenspace.
\end{lemma}
\begin{proof}
Suppose that there is an additional part in the $0$-eigenspace.  Since this cannot be the $1$- or $\lambda$-part it must be a $q^\iota$-part for some $q \in P_\alpha$, $\iota = \pm$.

Let $x$ be a $0$-eigenvector with decomposition into parts given by $x = x_0 + x_{\lambda-\frac{1}{2}}$, where $x_f \in A_f$ for $f = 0, \lambda-\frac{1}{2}$ and $y \in A_{q^\iota}$.  Similarly, let $z$ be a $\frac{1}{2}$-eigenvector with decomposition $z = z_\lambda +z_{p^\epsilon}$, where $z_f \in A_f$ for $f = \lambda, p^\epsilon$.  We will show that we can distinguish the product $xz$ from $yz$ and hence distinguish the sum of the $0$- and $(\lambda-\frac{1}{2})$-parts.

From the fusion law, we have $0\star \lambda  = \emptyset= (\lambda-\frac{1}{2}) \star \lambda$.  Easy calculations from the previous proof show that $0 \star p^\epsilon = p^\epsilon$ and $(\lambda-\frac{1}{2})\star p^\epsilon = p^\epsilon$ (\cite[Lemmas 4.6 and 4.11]{CM19}).  So, $(x_0+x_{\lambda+\frac{1}{2}})z$ is a $\frac{1}{2}$-eigenvector.

By \cite[Lemma 4.9]{CM19}, $q^\iota \star \lambda$ is exactly equal to $\emptyset$ if $q = (0, |\alpha|)$; to $q^{-\iota}$ if $\xi_\beta = 0$, where $p(\beta)=q$; and to $\{q^\iota, q^{-\iota}\}$ otherwise. Note that, since $\theta_\beta^\pm$ are distinct roots, $\nu_r^+ \neq \nu_r^-$, for any $r \in P_\alpha$.  So, in the last case, $\phi(q^\iota) \neq \phi(q^{-\iota})$ and hence $yz$ has a non-trivial projection onto an eigenspace other than the $\frac{1}{2}$-eigenspace and therefore we can distinguish $xz$ from $yz$.

We now consider the first two cases.  Note that, since $\phi(q^\iota) = 0$, we have $0 = \nu_q^\iota = \frac{1}{4} + \mu_\alpha b_{\alpha, \beta}(\theta_\beta^\iota + \xi_\beta)$, where $p(\beta) = q$.

For the first case, $q = (0, |\alpha|)$ and so $\xi_\beta = \frac{1}{4\mu_\alpha b_{\alpha, \beta}}$.  Substituting this above, we get $0 = \frac{1}{2} + \mu_\alpha b_{\alpha, \beta}\theta_\beta^\iota$ and so $\theta_\beta^\iota = -\frac{1}{2\mu_\alpha b_{\alpha, \beta}}$.  However, $\theta_\beta^\iota$ is a root of $x^2 + 2\xi_\beta +1 = 0$, which is a contradiction. 

For the second case, we get a contradiction unless $\phi(q^{-\iota}) = \frac{1}{2}$.  So, we assume that $\phi(q^{-\iota}) = \frac{1}{2}$ and $\xi_\beta=0$.  In particular, $w_\beta^\iota := \theta_\beta^\iota e^\beta + e^{\alpha+\beta}$ is a $0$-eigenvector and $w_\beta^{-\iota}$ is a $\frac{1}{2}$-eigenvector, where $p(\beta) = q$.  An easy calculation (\cite[Lemma 4.12]{CM19}) shows that
\[
w_\beta^\iota w_\beta^\kappa = (\theta_\beta^\iota \theta_\beta^\kappa) c_\beta t_\beta + c_{\alpha+\beta}t_{\alpha+\beta} + b_{\beta, \alpha+\beta}(\theta_\beta^\iota +  \theta_\beta^\kappa)e^\alpha
\]
for $\kappa = \pm$.  Since $\xi_\beta = 0$, we have $w_\beta^\iota w_\beta^{-\iota} = c_\beta t_\beta + c_{\alpha+\beta}t_{\alpha+\beta}$.  So for $y = w_\beta^\iota$ not to be distinguished, $c_\beta t_\beta + c_{\alpha+\beta}t_{\alpha+\beta}$ must be a $\frac{1}{2}$-eigenvector.  In particular, from Table \ref{tab:esp}, we see that it must be in the $\lambda$-part and hence $c_{\alpha+\beta} = -c_\beta$ and $|\beta| = \frac{|\alpha|}{2}$.  However, $y^2 = (w_\beta^\iota)^2 = (\theta_\beta^\iota)^2 c_\beta t_\beta + c_{\alpha+\beta}t_{\alpha+\beta} + 2 b_{\beta, \alpha+\beta}\theta_\beta^\iota e^\alpha$ which must be a $0$-eigenvector.  In particular, it does not contain any $\lambda$-part and so each $t_i$ for $i \in \supp(\alpha)$ has the same coefficient.  Since $\xi_\beta = 0$, we have $\theta_\beta^\iota = \pm 1$.  Also, since $|\beta| = \frac{|\alpha|}{2}$, both $t_\beta$ and $t_{\alpha+\beta}$ intersect $t_\alpha$ and hence $c_{\alpha+\beta} = (\theta_\beta^\iota)^2c_\beta = c_\beta$.  Since $c_\beta \neq 0$, this is a contradiction.
\end{proof}

Therefore, given any $e_+$ we can always pair it with $e_-$.  So, we have the following.

\begin{corollary}
We can partition $X$ into pairs $\{\lambda t_\alpha \pm \mu_\alpha e^\alpha \}$.
\end{corollary}

This completes the proof of Theorem \ref{recover}.


We now turn our attention to the $c$ structure constants in the algebra.

\begin{theorem}\label{csame}
Suppose $C$ is a projective code and $A_C$ is a $\mathbb{Z}_2$-graded code algebra satisfying the Axis Hypothesis for $S$. If $C \neq \mathbb{F}_2^2$, then, $c := c_{\alpha}$ for all $\alpha \in S$.
\end{theorem}
\begin{proof}
Consider the graph on $S$ with edges $\alpha \sim \beta$ if $\beta \neq \alpha^c$.  Since $C$ is projective and $C \neq \mathbb{F}_2^2$, this graph is connected.  Hence, it suffices to show that $c_\alpha = c_\beta$ for $\alpha, \beta \in S$ with $\beta \neq \alpha^c$.

We fix notation.  Let $p^\epsilon(\alpha)$ denote the $p^\epsilon$-part with respect to the small idempotent associated to $\alpha$.  Similarly, for the coefficients $\xi_\beta(\alpha)$ and $\theta^\epsilon_\beta(\alpha)$.

By Theorem \ref{recover}, the parts of $\lambda t_\alpha + \mu_\alpha e^\alpha$ and $\lambda t_\beta + \mu_\beta e^\beta$ are known.  Let $p = p(\beta)$ be the weight partition of $\beta$ with respect to $\alpha$.  In other words, $\beta \in C_\alpha(p)$.  Then, $\alpha \in C_\beta(p)$ for the same weight partition $p$.  So $p^\epsilon(\alpha)$ and $p^\epsilon(\beta)$ both have the same eigenvalue.
\begin{align}
0 &= \phi(p^\epsilon(\alpha)) - \phi(p^\epsilon(\beta))\nonumber\\
&= b_{\alpha, \beta}\left( \mu_\alpha \theta^\epsilon_\beta(\alpha) - \mu_\beta \theta^\epsilon_\alpha(\beta) + \mu_\alpha \xi_\beta(\alpha)  - \mu_\beta \xi_\alpha(\beta)  \right)\label{eqn:eigeq}
\end{align}
Note that $b_{\alpha, \beta} \neq 0$.  We split into two cases: either $|\alpha| - 2|\alpha \cap \beta| = 0$ in $\mathbb{F}$, or not.

Suppose that $|\alpha| - 2|\alpha \cap \beta| = 0$.  By \cite[Lemma 3.3]{CM19}, this holds if and only if $\xi_\beta(\alpha) = 0$ which is equivalent to $\theta^\epsilon_\beta(\alpha) = \pm 1$.  In particular, we may choose our labelling so that $\theta^+_\beta(\alpha) = 1$.  Note that $|\beta| - 2|\alpha \cap \beta| = 0$ too and so we have the analogous result with respect to $\beta$.  In this case, Equation \ref{eqn:eigeq} yields $\mu_\alpha = \mu_\beta$ and hence $c_\alpha = c_\beta$.

Now, suppose that $|\alpha| - 2|\alpha \cap \beta| = |\beta| - 2|\alpha \cap \beta| \neq 0$.  By rearranging the equations for $\xi_\beta(\alpha)$ and $\xi_\alpha(\beta)$ and combining, we obtain
\[
\xi_\beta(\alpha) = \tfrac{\mu_\beta}{\mu_\alpha} \xi_\alpha(\beta) 
\]
So Equation \ref{eqn:eigeq} yields $\theta^\epsilon_\beta(\alpha) = \frac{\mu_\beta}{\mu_\alpha} \theta^\epsilon_\alpha(\beta)$.  We write $\omega := \frac{\mu_\beta}{\mu_\alpha}$.  By definition, $\theta^\pm_\alpha(\beta)$ are the two solutions to
\[
x^2 + 2\xi_\alpha(\beta) x -1 =0
\]
So, $\theta^\pm_\beta(\alpha) = \omega \theta^\pm_\alpha (\beta)$ are the two solutions to
\[
x^2 + 2\omega \xi_\alpha(\beta) x -\omega^2 = x^2 + 2 \xi_\beta(\alpha) x -\omega^2 = 0
\]
However, $\theta^\pm_\beta(\alpha)$ are already the two solutions to $x^2 + 2 \xi_\beta(\alpha) x -1 = 0$ and hence $\omega^2=1$.  That is, $\mu_\alpha = \mu_\beta$ (again using choices of sign for roots) and hence $c_\alpha = c_\beta$.
\end{proof}

In particular, except for case (1a) where $C = \mathbb{F}_2^2$ and $a = -1$, all the $\mathbb{Z}_2$-graded algebras in Theorem \ref{Z2grading} have $c_\alpha := c_{\alpha_i} = c_{\alpha_j}$ for $\alpha_i, \alpha_j \in S$ and hence $\mu := \mu_{\alpha_i} = \mu_{\alpha_j}$ also.

\section{Automorphisms}\label{sec:auts}

In this section, we calculate the automorphism groups for each of the $\mathbb{Z}_2$-graded algebras in Theorem \ref{Z2grading}.  Throughout this section, let $S$ be a set of codewords and $X$ be the corresponding pair-closed set of small-idempotents.  First note that, in all but case (1a), the $c$ structure constant is the same for all $\alpha \in S$ and we write $c = c_\alpha$ for this.  We also write $\mu := \mu_\alpha$ when $C \neq \mathbb{F}_2^2$.

For cases (2) and (3) in Theorem \ref{Z2grading}, the code $D = \proj_\alpha(C)$ has a codimension one subcode $D_+$  which is the sum of weight sets and $\1 \in D_+$.  Let $D_- := D \setminus D_+$.  Then in both these cases, the negatively graded part is given by
\[
A_- = \bigoplus_{m \in wt(D_-)} A_{(m, |\alpha|-m)^\pm}
\]
Note that since $1 \in D_+$, $m \in wt(D_\pm)$ if and only if $|\alpha|-m \in wt(D_\pm)$, so the indexing is indeed over parts $p^\epsilon$.  We will abuse notation and just write $p(\beta) \in wt(D_\pm)$ for $m \in wt(D_\pm)$, where $p(\beta) = (m, |\alpha|-m)$.  We also observe that the grading of the part $p^\epsilon$ depends only on the partition $p$, not $\epsilon$. We call this the \emph{standard case}.

As noted, in case (2), provided we choose the structure constants in a `nice' way, the $\mathbb{Z}_2$-grading extends to a $\mathbb{Z}_2 \times \mathbb{Z}_2$-grading.  In this situation, we also describe the extra Miyamoto automorphisms obtained in addition to those in the standard case.

The graded parts in cases (1a) and (1b) in Theorem \ref{Z2grading} are different from the other two cases. Here the negative part is not described by a codimension one code, so we must deal with these cases separately.

\subsection{The standard case}

In this case, $|\alpha| \ge 2$ and let $e = e_{\alpha, \pm}$.

\begin{proposition}\label{tauaction}
The action of $\tau_e$ on $A$ is given by
\begin{align*}
t_i &\mapsto t_i \\
e^\beta &\mapsto \begin{cases}
e^\beta & \mbox{if } p(\beta) \in  wt(D_+) \\
-e^\beta& \mbox{if } p(\beta) \in  wt(D_-)
\end{cases}
\end{align*}
\end{proposition}
\begin{proof}
By definition $1,0, \lambda - \frac{1}{2}$ and $\lambda$ (where it exists) are in the positive part of the grading.  Since $t_i$, $e^\alpha$ and $e^{\alpha^c}$ are all in the span of $A_+$, they are all fixed by $\tau_e$.  Now consider $\beta \in C^* \setminus \{ \alpha, \alpha^c\}$.  We may write
\[
(\theta^+_\beta - \theta^-_\beta) e^\beta = (\theta^+_\beta e^\beta + e^{\alpha+\beta})  - ( \theta^-_\beta e^\beta + e^{\alpha + \beta}) = w^+_\beta - w^-_\beta
\]
Recall that by the Field Hypothesis, $\theta^+_\beta \neq \theta^-_\beta$.  Since the grading of $p^\pm$ depends only on the partition $p$ and not on $\pm$ and $\tau_e$ negates the $w^\pm_\beta$ for $\beta$ such that $p(\beta)$ is graded negatively, the result follows.
\end{proof}

\begin{corollary}
The automorphism $\tau_e$ does not depend on the value of the structure constants.
\end{corollary}

Recall that for a given $\alpha \in C$, we have two different idempotents $e_\pm := \lambda t_\alpha \pm \mu e^\alpha$.

\begin{corollary}\label{plusminussame}
$\tau_{e_+} = \tau_{e_-}$.
\end{corollary}
\begin{proof}
By Proposition \ref{tauaction}, we need only consider the $e^\beta$.  In particular, whether or not $e^\beta$ is in the positive or negative part depends only on its intersection with $\alpha$, not the value of $\mu$.
\end{proof}

In particular, $A$ is an axial decomposition algebra where the $\tau$-map, which is the map on the axes given by $a \mapsto \tau_a$, is not a bijection. It is, in general, a $2$-fold cover.

 In light of the above result, we will write $\tau_\alpha$ for $\tau_{e_\pm}$, where $e = \lambda t_\alpha \pm \mu e^\alpha$.

We may now consider the Miyamoto group in the case where the grading is a $\mathbb{Z}_2$-grading, but not a $\mathbb{Z}_2 \times \mathbb{Z}_2$-grading.

\begin{corollary}\label{Miymod}
In the standard case, the Miyamoto group is an elementary abelian $2$-group of order at most $2^{|S|}$.
\end{corollary}
\begin{proof}
Consider the decompositions into positive and negative parts for each Miyamoto involution $\tau_e$, $e \in X$.  By the above proposition, we see that the standard basis of $t_i$, $i = 1, \dots, n$ and $e^\alpha$, $\alpha \in C^*$ is a refinement of the intersection of all the decompositions given by the gradings for each Miaymoto involution $\tau_e$.  In particular, since each $\tau_e$ acts as $\pm 1$ on each basis element, it is clear that any two Miyamoto involutions commute and hence $G$ is abelian.  Since all the $\tau_e$ have order two, $G$ is an elementary abelian $2$-group.  By Corollary \ref{plusminussame}, it is clear that there are at most $|S|$ generators, hence the order is at most $2^{|S|}$.
\end{proof}

In the case where $|\alpha|=2$, we can say more.  In this case, $C = \bigoplus_{i=1}^r C_i$ is the direct sum of even weight codes $C_i$ all of length $m$.

\begin{corollary}\label{Nmod}
Let $|\alpha|=2$ and so $C = \bigoplus_{i=1}^r C_i$, where $C_i$ has length $m$.  Then the Miyamoto group is
\[
\la \tau_\alpha : \alpha \in S \ra \cong
\begin{cases}
2^{r(m-1)} & \mbox{if $m$ is odd}\\
2^{r(m-2)} & \mbox{if $m$ is even}
\end{cases}
\]
\end{corollary}
\begin{proof}
First, by Corollary \ref{Miymod}, $\tau_\alpha$ and $\tau_\beta$ commute, where $\alpha, \beta \in S$ and $N$ is an elementary $2$-group of order at most $2^{|S|}$.  We consider the case where $C$ is an irreducible code; the general case follows immediately from this.

Define a map $\psi \colon C \to N$ by $\alpha \mapsto \tau_\alpha$ for $\alpha \in S$ and extend linearly.

We first show that $\psi$ is a well-defined homomorphism.  Let $\alpha \in C$ and decompose $\alpha$ in two different ways as $\alpha = \alpha_1 + \dots + \alpha_k$ and $\alpha = \beta_1 + \dots + \beta_l$, where $\alpha_i, \beta_j \in S$.  We compare the actions of $\tau_{\alpha_1} \dots \tau_{\alpha_k}$ and $\tau_{\beta_1} \dots \tau_{\beta_l}$.  Both fix all the $t_i$.  The codeword element $e^\gamma$ is inverted by some $\tau_\delta$, $\delta \in S$, if and only if $|\gamma \cap \delta| = 1$.  Now, if $|\gamma \cap \alpha|$ equals $0$, respectively $1$, then $\gamma$ intersects an even, respectively an odd, number of $\alpha_i$.  However, the same is true for the $\beta_j$.  Hence $\tau_{\alpha_1} \dots \tau_{\alpha_k} = \tau_{\beta_1} \dots \tau_{\beta_l}$ and $\psi$ is a well-defined homomorphism.

Since $\psi$ is a homomorphism, $N$ has order at most $2^{\dim(C)}$.  Suppose that $\tau_\alpha = 1$ for some $\alpha \in C \setminus \0$.  Then every $\beta$ must intersect $\alpha$ in a set of size $0$.  Observe that the parity of the intersection of two codewords is equal to their inner product.  It is clear that the only possible vector which has inner product $0$ with every codeword in an irreducible even weight code is $\1$.  However, $\1 \in C$ if and only if $m$ is even.  So, when $m$ is even, the kernel of $\psi$ is $\la \1 \ra$, otherwise it is trivial.
\end{proof}

\begin{remark}
Note that the above proof does not hold when $|\alpha| \neq 2$.  When $|\alpha| \neq 2$, the map $\phi$ may not be well-defined.
\end{remark}

\begin{lemma}\label{stdclosed}
The orbits of $G(X)$ on $X$ are the pairs $\{ e_{\alpha,+}, e_{\alpha,-}\}$ for $\alpha \in S$.  In particular, the set of axes $X = \{ e_{\alpha, \pm} : \alpha \in S\}$ is closed.
\end{lemma}
\begin{proof}
Let $\alpha, \beta \in S$.  By Lemma \ref{tauaction}, $ \tau_{\beta}$ fixes $e_{\alpha, +}$ and $e_{\alpha, -}$ if $p(\alpha) \in D_+$ and swaps $e_{\alpha, +}$ and $e_{\alpha, -}$ if $p(\alpha) \in D_-$.
\end{proof}

\subsection{$|\alpha|=1$}

We now deal with the case where $|\alpha|=1$ as this behaves differently to the standard case. There are two subcases here where $A$ is $\mathbb{Z}_2$-graded.  Either $n=3$, or $n=2$ and $a=-1$.

\subsubsection{$n=2$ and $a=-1$}

Note that, as $C \cong \mathbb{F}_2^2$, $A$ decomposes as:
\[
A = A_1 \oplus A_2 = \langle t_1, e^{\alpha_1} \rangle \oplus \langle t_2, e^{\alpha_2} \rangle
\]
In particular, there are four small idempotents
\[
e_{i, \pm} := \lambda t_i \pm \mu_i e^{\alpha_i} \in A_i
\]
where $\lambda = \frac{1}{2a|\alpha_i|} = -\frac{1}{2}$ and $\mu_i^2 = \frac{\lambda - \lambda^2}{c_{\alpha_i}} = -\frac{3}{4 c_{\alpha_i}}$ for $i = 1,2$.  However, since this is the exception to Theorem \ref{csame}, $c_{\alpha_1}$ and $c_{\alpha_2}$ may differ, and so $\mu_1$ and $\mu_2$ need not be the same either. In fact, this does produce a valid $\mathbb{Z}_2$-graded algebra with a non-trivial Miyamoto group.

The fusion law has three parts $1$, $0$ and $\lambda - \frac{1}{2}$, where $\lambda - \frac{1}{2}$ is in the negatively graded part and $1$ and $0$ are in the positive part.  Since $\lambda = \frac{1}{2}$, $\phi(\lambda - \frac{1}{2}) = -1$ and all the parts are distinct.  Note that the $(\lambda - \frac{1}{2})$-part is $1$-dimensional and is spanned by $2\mu_i c_{\alpha_i} t_i - e^{\alpha_i}$.  We write $\tau_{i, \pm}$ for $\tau_{e_{i, \pm}}$.

\begin{lemma}\label{C2act}
Let $i,j = 1,2$ with $i \neq j$.  The automorphism $\tau_{i, \pm}$ acts trivially on $A_j$ and acts on $A_i = \langle t_i, e^{\alpha_i} \rangle$ as
\[
\begin{pmatrix}
-\frac{1}{2} & \mp \mu_i \\
\mp \frac{3}{4 \mu_i} & \frac{1}{2}
\end{pmatrix}
\]
\end{lemma}
\begin{proof}
We have $\langle \lambda t_i + \mu_i e^{\alpha_i}, 2\mu_i c_{\alpha_i} t_i - e^{\alpha_i} \rangle = \langle t_i, e^{\alpha_i} \rangle = A_i$ and so $A_i$ is spanned by the $1$- and $-1$-eigenspaces.  In particular, $\tau_{i, \pm}$ must act trivially on $A_j$.  One can check that $\tau_{i, \pm}$ acts on $A_i$ as given by the matrix above by checking the action on the two vectors $\lambda t_i + \mu_i e^{\alpha_i}$ and $2\mu_i c_{\alpha_i} t_i - e^{\alpha_i}$.
\end{proof}

Note that in this case, unlike the Miyamoto involution in the standard case, $\tau_{i,\pm}$ does depend on $\mu_i$ and hence on the structure constant $c_{\alpha_i}$.  Also, $\tau_{i,+} \neq \tau_{i,-}$.

\begin{lemma}
The subgroup $G_i := \langle \tau_{i,+}, \tau_{i,-} \rangle \cong S_3$.
\end{lemma}
\begin{proof}
The automorphisms $\tau_{i,+}$ and $\tau_{i,-}$ are involutions so they generate a dihedral group.  Calculating the action of $\tau_{i,+} \tau_{i,-}$ on $A_i$, we have
\[
\begin{pmatrix}
-\frac{1}{2} & - \mu_i \\
- \frac{3}{4 \mu_i} & \frac{1}{2}
\end{pmatrix}
\begin{pmatrix}
-\frac{1}{2} & \mu_i \\
\frac{3}{4 \mu_i} & \frac{1}{2}
\end{pmatrix}
 =
 \begin{pmatrix}
-\frac{1}{2} & - \mu_i \\
\frac{3}{4 \mu_i} & -\frac{1}{2}
\end{pmatrix}
\]
It is straightforward to check that this has order three.
\end{proof}

Note that although the action of $G_i$ on the space $A_i$ does depend on $\mu_i$ and so on the structure constant $c_{\alpha_i}$, the isomorphism type of the group $G_i$ does not.

\begin{corollary}
The Miyamoto group is isomorphic to $S_3 \times S_3$.
\end{corollary}

\begin{lemma}
The orbit of $e_{i, \pm}$ under $G_i$ is $\{ t_i, e_{i,+}, e_{i,-} \}$.  Hence the closure of $\{ e_{i,\pm} : i = 1,2\}$ is $\{t_1, t_2, \} \cup \{ e_{i,\pm} : i = 1,2\}$.
\end{lemma}
\begin{proof}
By Lemma \ref{C2act}, it is easy to see that the involution $\tau_{i,\pm}$ maps $t_i$ to $e_{i, \mp}$.  Since by definition it fixes $e_{i,\pm}$ and $G_i \cong S_3$, the orbit of $e_{i, \pm}$ is $\{ t_i, e_{i,+}, e_{i,-} \}$.
\end{proof}

\subsubsection{$n=3$}

In this case, $e$ has five parts, $1$, $0$, $\lambda - \frac{1}{2}$ and $p^\pm$, where $p = (0,1)$ is the only partition.  Here $p^\pm$ are in the negative part and $1$, $0$, $\lambda - \frac{1}{2}$ are in the positive part.

\begin{lemma}
We have $\tau_{e_+} = \tau_{e, -}$ and the Miyamoto group is the Klein four group $2^2$.
\end{lemma}
\begin{proof}
In a similar way to Proposition \ref{tauaction}, we see that $\tau_{\alpha, \pm}$ acts trivially on the $t_i$ and on $e^\alpha$ and $e^{\alpha^c}$ and inverts all the $e^\beta$ where $\beta \in C^* -\{ \alpha, \alpha^c\}$.  Hence we have $\tau_{\alpha, +} = \tau_{\alpha, -}$.  

Let $\alpha_i$ be the codeword with a $1$ in the $i$\textsuperscript{th} position and $0$s elsewhere.  So, $S = \{ \alpha_i : i = 1,2,3 \}$.  Note that $C^* = S \cup S^c$, where $S^c = \{ \alpha^c : \alpha \in S \}$.  Let $\{i,j,k\} = \{1,2,3\}$.  Now, $\tau_{\alpha_i}$ inverts all $e^\beta$ where $\beta \in C^* -\{ \alpha_i, \alpha_i^c\}$.  In particular, nothing is fixed by both $\tau_{\alpha_i}$ and $\tau_{\alpha_j}$ and only $e^{\alpha_k}$ and $e^{\alpha_k^c}$ are inverted by both.  In other words, $\tau_{\alpha_i}\tau_{\alpha_j}$ inverts all $e^\beta$ where $\beta \in C^* -\{ \alpha_k, \alpha_k^c\}$.  Hence $\tau_{\alpha_i}\tau_{\alpha_j} = \tau_{\alpha_k}$ and the Miyamoto group is a Klein four group.
\end{proof}

\begin{lemma}
The orbits of $G(X)$ on $X$ are the pairs $\{ e_{\alpha,+}, e_{\alpha,-}\}$ for $\alpha \in S$.  In particular, the set of axes $X = \{ e_{\alpha, \pm} :  \alpha \in S\}$ is closed.
\end{lemma}
\begin{proof}
The proof is analogous to the proof of Lemma \ref{stdclosed}.
\end{proof}


\subsection{The $\mathbb{Z}_2 \times \mathbb{Z}_2$-graded case}\label{sec:Z2}

In this case, $|\alpha| = 2$ and $C := \bigoplus_{i=1}^r C_i$, where the $C_i$ all have length $m$.  From the previous paper we have:

\begin{proposition}\textup{\cite[Proposition 5.10]{CM19}}\label{Z2xZ2}
Let $C =\bigoplus_{i=1}^r C_i$ be the direct sum of even weight codes $C_i$ all of length $m$, $n = m^r$ such that $n \geq 5$ and $m \geq 3$.  Assume $|\alpha|= 2$ for $\alpha \in S$ and let $A_C$ be a code algebra satisfying the Axis Hypothesis with respect to $S$.  Assume further that $b_{\beta, \gamma} = b_{\alpha+\beta, \gamma}$ and $c_\beta = c_{\alpha+\beta}$ for all $\alpha \in S$, $\beta \in C_{\alpha}(1)$ and $\gamma \in C^*\setminus \{\alpha, \alpha^c\}$.  Then, the axial decomposition algebra $A_C$ has a fusion law given by Table $\ref{tab:al2special}$ and has a $\mathbb{Z}_2 \times \mathbb{Z}_2$-grading given by
\begin{align*}
A_{(0,0)} &= A_1 \oplus A_0 \oplus A_{\lambda - \frac{1}{2}} \oplus A_{p_0^+} \oplus A_{p_0^-}  \\
A_{(1,0)} &= A_{p_1^+}  \\
A_{(0,1)} &= A_{p_1^-}  \\
A_{(1,1)} &= A_{\lambda}
\end{align*}

\begin{table}[!htb]
\setlength{\tabcolsep}{7pt}
\renewcommand{\arraystretch}{1.5}
\centering
\begin{tabular}{c|ccccccccc}
 & $1$ & $0$ & $\lambda$ & $\lambda-\frac{1}{2}$ & ${p_1^+}$ & ${p_1^-}$ & ${p_0^+}$ & ${p_0^-}$ \\ \hline
$1$ & $1$ &  & $\lambda$ & $\lambda-\frac{1}{2}$ & ${p_1^+}$ & ${p_1^-}$ & ${p_0^+}$ & ${p_0^-}$  \\
$0$ &  & $0$  & &    & ${p_1^+}$ & ${p_1^-}$ & ${p_0^+}$ & ${p_0^-}$ \\
$\lambda$ & $\lambda$ & & $1, \lambda - \frac{1}{2}$ & & ${p_1^-}$  & ${p_1^+}$   \\ 
$\lambda-\frac{1}{2}$ & $\lambda-\frac{1}{2}$& &   & $1, \lambda-\frac{1}{2}$ & ${p_1^+}$ & ${p_1^-}$ & ${p_0^+}, {p_0^-}$ &  ${p_0^+}, {p_0^-}$ \\
${p_1^+}$ & ${p_1^+}$ & ${p_1^+}$  & ${p_1^-}$  &  ${p_1^+}$ & $X$ & $\lambda$ &  ${p_1^+} $ &  ${p_1^+}$ \\
${p_1^-}$ & ${p_1^-}$ & ${p_1^-}$  &  ${p_1^+}$ &  ${p_1^-}$  & $\lambda$  & $X$ &  ${p_1^-} $ &  ${p_1^-} $ \\
${p_0^+}$ & ${p_0^+}$ & ${p_0^+}$  & & ${p_0^+}, {p_0^-}$ & $ {p_1^+}$ & $ {p_1^-} $ & $X$ & $X$ \\
${p_0^-}$ & ${p_0^-}$ & ${p_0^-}$  &  & ${p_0^+}, {p_0^-}$ & $ {p_1^+}$ & ${p_1^-} $  & $X$ & $X$
\end{tabular}
\vspace{5pt}\\
where $X = 1,0, \lambda - \frac{1}{2}, {p_0^+}, {p_0^-}$
\caption{Fusion law for $\vert \alpha \vert = 2$}
\label{tab:al2special}
\end{table}

\end{proposition}

\begin{remark}
We note that although there are three parts which are non-trivially graded, one of them is distinguishable from the rest.  Indeed, $A_\lambda$ is $1$-dimensional, whereas both $A_{p_1^+}$ and $A_{p_1^-}$ have dimension $\frac{|C|}{4}  \gneqq 1$.
\end{remark}

If $|\alpha \cap \beta| = 1$, then $\xi_\beta = 0$ and $\theta^\pm_\beta = \pm 1$.  So, $A_{p_1^+}$ is spanned by $w^+_\beta = e^\beta + e^{\alpha+\beta}$ and $A_{p_1^-}$ is spanned by $w^-_\beta = -e^\beta + e^{\alpha + \beta}$.

We write $T := \mathbb{Z}_2 \times \mathbb{Z}_2$.  Recall that we have a map $\tau \colon X \times T^* \to \Aut(A)$ where $\tau_a(\chi)$ acts by
\[
v \mapsto \chi(t) v
\]
for $v \in A_t$.  The non-trivial characters can be labelled by non-trivial elements of $T$ and they are given by
\[
\chi_t : s \mapsto \begin{cases} 1 & \mbox{if } s = 1_T, t \\
-1 & \mbox{otherwise} \end{cases}
\]

In light of the above Proposition \ref{Z2xZ2}, we abuse notation and label our non-trivial characters by $\chi_+$, $\chi_-$ and $\chi_\lambda$ corresponding to the grading of those parts.

\begin{lemma}
Let $e_\pm := \lambda t_\alpha \pm \mu e^\alpha$.  Then, for all $\chi \in T^*$
\[
\tau_{ e_+}(\chi) = \tau_{ e_-}(\chi)
\]
\end{lemma}
\begin{proof}
Note that $e_{\alpha, +}$ and $e_{\alpha, -}$ differ just by the parity of their coefficient $\mu$ of $e^\alpha$.  Observe that $A_\lambda = \langle t_i - t_j \rangle$, where $\supp(\alpha) = \{ i,j\}$, does not depend on $\mu$.  Similarly, $A_{p_1^\pm} = \langle \pm e^\beta + e^{\alpha + \beta} \rangle$ which do not depend on $\mu$.
\end{proof}

In particular, just as in the standard case, the $\tau$-map given by $a \mapsto \tau_a(\chi)$ for any $\chi \in T^*$ is not a bijection.  Similarly to the standard case, we will write $\tau_\alpha(\chi)$ for $\tau_{e_\pm}(\chi)$, where $e_\pm = \lambda t_\alpha \pm \mu e^\alpha$.

\begin{lemma}\label{Z2extraact}
Let $\supp(\alpha) = \{ i,j\}$.  The action of $\tau_{\alpha}( \chi_\pm)$ is given by
\begin{align*}
\tau_{\alpha}(\chi_\pm) : t_i &\mapsto t_j \\
t_j & \mapsto t_i \\
t_k& \mapsto t_k \qquad \mbox{for } k \notin \supp(\alpha) \\
e^\beta &\mapsto \begin{cases}
e^\beta & \mbox{if } \beta \in C_{\alpha}(0) \\
\pm e^{\alpha+\beta} & \mbox{if } \beta \in C_{\alpha}(1)
\end{cases}
\end{align*}
\end{lemma}
\begin{proof}
Note that $\tau_{\alpha}( \chi_\pm)$ inverts the $\lambda$-part, which is spanned by $t_i - t_j$, and the $p_1^{\mp}$-part and fixes all the other parts.  For $\beta \in C_\alpha(1)$, by writing $e^\beta = \frac{1}{2}\left( (e^\beta + e^{\alpha + \beta}) - (-e^\beta + e^{\alpha+\beta})\right)  = \frac{1}{2} ( w^+_\beta - w^-_\beta)$ and applying $\tau_{\alpha}( \chi_\pm)$, we see the result follows.
\end{proof}

 It is easy to see from Lemma \ref{Z2extraact} that $\tau_\alpha( \chi_+)$ acts as the involution $(i,j)$ in the natural action, where $\supp(\alpha) = \{ i,j\}$.  Hence $\langle \tau_\alpha(\chi_+) : \alpha \in S \rangle \cong \Aut(C)$.  From \cite[Lemma $3.6$]{codealgebras}, $\Aut(C)$ acts on $A$ in the natural action if and only the structure constants are regular.  That is, $a_{i,\alpha} = a_{i^g, \alpha^g}$, $b_{\alpha, \beta} = b_{\alpha^g, \beta^g}$ and $c_{i,\alpha} = c_{i^g, \alpha^g}$ for all $g \in \Aut(C)$, $i \in \supp(\alpha)$, $\alpha, \beta \in C^*$, $\beta \neq \alpha, \alpha^c$.

\begin{corollary}
 The group $\langle \tau_\alpha(\chi_+) : \alpha \in S \rangle = \Aut(C)$ and the structure constants in $A$ are in fact regular.
\end{corollary}

We begin by considering the case where $C$ is an irreducible code. So, $\Aut(C) = \langle \tau_\alpha( \chi_+) : \alpha \in S \rangle \cong S_n$.   Note that $\tau_\alpha(\chi_\lambda)$ are just the Miyamoto automorphisms from the standard case.  Hence, by Corollary \ref{Nmod},
\[
N := \la \tau_\alpha(\chi_\lambda) : \alpha \in S \ra \cong
\begin{cases}
2^{n-1} & \mbox{if $n$ is odd}\\
2^{n-2} & \mbox{if $n$ is even}
\end{cases}
\]

\begin{lemma}
For axes $a,b \in X$ and $t \in T$, we have
\[
\tau_a(\lambda)^{\tau_b( \chi_t)} = \tau_{a^{\tau_b( \chi_t)}}( \chi_\lambda)
\]
\end{lemma}
\begin{proof}
This is a straightforward calculation.
\end{proof}

\begin{corollary}
$N$ is a normal subgroup of the Miyamoto group $G(X)$.
\end{corollary}

\begin{theorem}\label{Ceven_irred}
Let $C$ be an irreducible code, then $G(X) \cong 2^{n-1} : S_n$ if $n$ is odd and $G(X) \cong 2^{n-2} : S_n$ if $n$ is even.
\end{theorem}
\begin{proof}
From the definition of the $\tau$ map, $\tau_e( \chi_-) = \tau_e( \chi_\lambda) \tau_e( \chi_+)$ and so $G(X)$ is generated by just $\tau_e( \chi_\lambda)$ and $\tau_e( \chi_+)$, for $e \in X$.  It is clear that $\langle \tau_e( \chi_+) : e \in X \rangle = \Aut(C) \cong S_n$ intersects $N$ trivially, hence the result follows.
\end{proof}

We may now consider the general case, where $C$ is not irreducible.  

\begin{corollary}
Let $C = \bigoplus_{i =1}^r C_i$ be a sum of even weight codes $C_i$, all of length $m \geq 3$.  Then,
\[
G(X) = \begin{cases}
\prod_{i = 1}^r 2^{m-1}:S_m & \mbox{if $m$ is odd}\\
\prod_{i = 1}^r 2^{m-2}:S_m & \mbox{if $m$ is even}
\end{cases}
\]
\end{corollary}
\begin{proof}
Since $C = \bigoplus_{i =1}^r C_i$, each $\alpha \in S$ is contained in some $C_i$. Define $X_i$ as the set of axes in $X$ which come from codewords in $C_i$.  We claim that $G(X)$ is a central product of $G(X_i)$.

For $i \neq j$, let $\tau_i(\chi)$ and $\tau_j(\psi)$ be the Miyamoto automorphisms associated to an axis in $X_i$ and in $X_j$ coming from codewords $\alpha_i$ and $\alpha_j$ respectively.  Since $i \neq j$, it is clear that their actions commute on the toral elements and codewords elements $e^\beta$ such that $\beta \in C_{\alpha_i}(0) \cap C_{\alpha_j}(0)$.

Since $\alpha_i$ and $\alpha_j$ are disjoint, $\beta \in C_{\alpha_j}(1)$ if and only if $\alpha_i + \beta \in C_{\alpha_j}(1)$.  Since $e^\beta \mapsto -e^\beta$ commutes with $e^\beta \mapsto \pm e^{\alpha_i + \beta}$ and addition of codewords is commutative, we see that $\tau_i(\chi)$ and $\tau_j(\psi)$ commute and hence the claim follows.

Since each $C_i$ is irreducible, by Theorem \ref{Ceven_irred}, $G(X_i)$ is isomorphic to either $2^{m-1}:S_m$, or $2^{m-2}:S_m$.  Since these are centre-free, the result follows.
\end{proof}

\end{document}